\documentclass[12pt]{article}
\usepackage{standalone}
\usepackage{amsmath, amssymb, amsthm}
\usepackage{indentfirst}
\usepackage{fullpage}
\usepackage{tikz}
\usepackage{bm}
\usepackage{float}
\usepackage{hyperref}
\usepackage{nicefrac}
\usetikzlibrary{shapes.geometric, shapes.misc, arrows, arrows.meta, decorations.markings, decorations.pathreplacing, calc, positioning}

\definecolor{Xcolor}{HTML}{FF0000}
\definecolor{Ycolor}{HTML}{00FFFF}

\DeclareMathOperator{\diam}{diam}

\DeclareMathOperator{\dis}{dis}
\DeclareMathOperator{\codis}{codis}

\newcommand\restr[2]{{
  \left.\kern-\nulldelimiterspace 
  #1 
  \vphantom{\big|} 
  \right|_{#2} 
  }}

\newcommand{\dH}{d_{\mathrm{H}}}
\newcommand{\dGH}{d_{\mathrm{GH}}}
\newcommand{\dmGH}{\widehat{d}_{\mathrm{GH}}}
\DeclareMathOperator*{\defeq}{\buildrel \mathrm{def}\over =}

\newtheorem{theorem}{Theorem}

\newtheorem{lemma}{Lemma}
\newtheorem{claim}{Claim}
\newtheorem*{corollary}{Corollary}
\theoremstyle{definition}

\theoremstyle{remark}
\newtheorem*{remark}{Remark}

\title{Lipschitz (non-)equivalence of the Gromov--Hausdorff distances, including on ultrametric spaces}
\author{Vladyslav Oles\thanks{National Center for Computational Sciences, Oak Ridge National Laboratory (\texttt{olesv@ornl.gov})} \and Kevin R. Vixie\thanks{Department of Mathematics and Statistics and Department of Physics and Astronomy, Washington State University (\texttt{krvixie@pm.me}, \url{https://iscilabs.com/})}}
\date{}
\begin{document}

\maketitle

\begin{abstract}
The Gromov–Hausdorff distance measures the difference in shape between compact metric spaces. While even approximating the distance up to any practical factor poses an NP-hard problem, its relaxations have proven useful for the problems in geometric data analysis, including on point clouds, manifolds, and graphs. We investigate the modified Gromov--Hausdorff distance, a relaxation of the standard distance that retains many of its theoretical properties, which includes their topological equivalence on a rich set of families of metric spaces. We show that the two distances are Lipschitz-equivalent on any family of metric spaces of uniformly bounded size, but that the equivalence does not hold in general, not even when the distances are restricted to ultrametric spaces. We additionally prove that the standard and the modified Gromov--Hausdorff distances are either equal or within a factor of 2 from each other when taken to a regular simplex, which connects the relaxation to some well-known problems in discrete geometry.

\end{abstract}

\section*{List of notation}
\def\arraystretch{1.5}
\vspace{-3mm}
\begin{table}[H]
\begin{tabular}{rp{0.8\textwidth}}
$|X|$ & the cardinality of set $X$\\
$X \times Y$ & Cartesian product of sets $X$ and $Y$\\
$d_X$ & the distance function of metric space $X$\\
$f^{-1}$ & the inverse of mapping $f$\\
$f \circ g$ & the composition of mappings $f$ and $g$ that applies $f$ to the output of $g$\\
$f(X)$ & the image of set $X$ under mapping $f$\\
$\dGH$ & the Gromov--Hausdorff distance\\
$\dmGH$ & the modified Gromov--Hausdorff distance
\end{tabular}
\end{table}
\def\arraystretch{1}

\section{Introduction}
The Gromov--Hausdorff (GH) distance measures the difference in shape between compact metric spaces. It was proposed by Gromov in the 1980s to study convergence of manifolds, and has appeared since then in many areas of modern mathematics, most notably Riemannian geometry and computational topology.

Computing the GH distance between metric spaces $X$ and $Y$ amounts to minimizing the distortion of their distances incurred by the choice of a pair $(f:X\to Y, g:Y\to X)$, which assigns to each point in $X$ and $Y$ at least one counterpart in the other space. While this formulation is resemblant of the Lawler's quadratic bottleneck assignment problem (QBAP), the GH distance is more expressive than QBAP due to the fact that the $(f, g)$-induced assignments do not need to be bijective. For example, Figure \ref{fig:uneven_sampling} shows that even when $|X| = |Y|$, non-bijective assignments can achieve smaller distortion than the bijective ones, making the Gromov--Hausdorff distance less sensitive to e.g. uneven sampling of points in $X$ and $Y$. Similarly to the NP-hard QBAP, computing the GH distance poses an intractable combinatorial optimization problem \cite{chazal2009gromov, memoli2007use, schmiedl2017computational}. As a consequence, the distance has been studied mainly from theoretical perspective, in particular for the topology it induces on the space of compact metric spaces. However, in the recent years a lot of attention has been put towards computational aspects of the GH distance, in particular in the context of shape matching and data analysis.

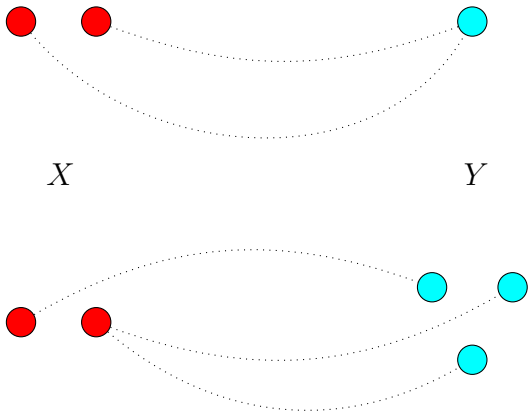
\begin{figure}[h]
    \centering
    \caption{An example with $|X| = |Y| = 4$ for which $(f, g)$-induced assignments minimizing the distortion are not bijective. Dotted lines represent one of such assignments.}
    \label{fig:uneven_sampling}
    \vspace{.5cm}
    \begin{tikzpicture}
\node[circle,draw,fill=Xcolor] (X1) at (0,0){};
\node[circle,draw,fill=Xcolor,left of=X1] (X2){};
\node[circle,draw,fill=Xcolor,below of=X1,node distance=4cm] (X3){};
\node[circle,draw,fill=Xcolor,left of=X3] (X4){};
\node[below left=1.6cm and 0cm of X1]{$X$};

\node[circle,draw,fill=Ycolor,right of=X1,node distance=5cm] (Y1){};
\node[circle,draw,fill=Ycolor,below left=3.25cm and .25cm of Y1] (Y2){};
\node[circle,draw,fill=Ycolor,below right=3.25cm and .25cm of Y1] (Y3){};
\node[circle,draw,fill=Ycolor,below of=Y1,node distance=4.5cm] (Y4){};
\node[below right=1.6cm and 4.6cm of X1]{$Y$};

\draw[dotted] (X1) to[out=-20,in=-160] (Y1);
\draw[dotted] (X2) to[out=-50,in=-120] (Y1);

\draw[dotted] (X4) to[out=30,in=160] (Y2);
\draw[dotted] (X3) to[out=-20,in=-150] (Y3);
\draw[dotted] (X3) to[out=-40,in=-150] (Y4);

\end{tikzpicture}
\end{figure}

Several bijection-based algorithms for approximating the distance have been proposed and applied for comparing point cloud \cite{bronstein2010gromov, villar2016polynomial} and graph data \cite{chung2017topological, lee2012persistent, fehri2018characterizing}. Grigor'yev et al. studied the GH distance between an arbitrary metric space and a regular simplex, and showed that various optimization problems in discrete geometry can be formulated in terms of this notion \cite{ivanov2016geometry, ivanov2019gromov}. Schmiedl proved that $\dGH(X, Y)$ cannot be approximated up to a factor of $\max\{|X|, |Y|\}^{\frac{1}{2}}$ (or up to a factor of 3 when restricted to ultrametric spaces) in polynomial time \cite{schmiedl2017computational}. Agarwal et al. proposed an algorithm for approximating the GH distance between metric trees up to a multiplicative factor, and showed that the problem is NP-hard if this factor is less than 3 \cite{agarwal2018computing}. 

Beside the exploration of the GH distance itself, several computationally motivated relaxations of it have been proposed.

\subsection{The Gromov--Wasserstein distances}
The family of Gromov--Wasserstein distances, defined between metric measure spaces, was proposed in \cite{memoli2011gromov} as a continuous relaxation of the GH distance inspired by the optimal transport problem. Here, $(f, g)$-induced assignments are replaced with soft assignments, where each pair in $X \times Y$ is associated with some positive weight and contributes to the resulting distortion proportionally to it. For every point, the weights of all the pairs which contain it must total to the measure of its singleton. This constraint guarantees that each point gets assigned its measure-worth of counterparts from the other metric measure space and contributes to distortion accordingly. However, it also implies that soft assignments generalize bijections rather than $(f, g)$-induced assignments, which reflects on the expressiveness of this relaxation. For example, under the natural choice of uniform probability measure in the example from Figure \ref{fig:uneven_sampling}, the optimal assignments are not allowed in the search space of the Gromov--Wasserstein distances as some points would get assigned twice their measure.

While computing the original Gromov--Wasserstein distances still poses an intractable problem of non-convex quadratic minimization, their regularized versions have proven to be computationally efficient for matching low-dimensional point clouds \cite{alvarez2018gromov, vayer2019sliced, bunne2019learning}, graphs \cite{xu2019gromov, xu2019scalable, vayer2018optimal, vayer2020fused}, and time series \cite{vayer2020fused}, as well as the associated tasks of vertex embedding \cite{xu2019gromov, vayer2018optimal, vayer2020fused}, graph partitioning \cite{xu2019scalable}, and computing geometric average of pairwise descriptor matrices \cite{peyre2016gromov, vayer2018optimal, vayer2020fused}. Additionally, computing a lower bound of the Gromov--Wasserstein distances was shown to be useful for discriminating graphs of up to several hundred vertices in \cite{chowdhury2019gromov}.

\subsection{The modified Gromov--Hausdorff distance}
The modified Gromov--Hausdorff (mGH) distance between compact metric spaces was proposed by M\'emoli in \cite{memoli2012some} as another relaxation and a lower bound of the standard GH distance. Instead of solving the minimization problem over the space of all possible pairs $(f, g)$, the mGH distance between metric spaces $X$ and $Y$ is computed by minimizing the distortion of assignment induced by the choice of $f:X\to Y$ and, separately --- of $g:Y \to X$. Each of the two assignments involves all points in one metric space but possibly only a subset of points in the other one.

The decoupled optimization problems are more amenable to direct computation, although remain to be of exponential size. Unlike the standard GH distance, the mGH distance allows for a rich family of lower bounds. These lower bounds were leveraged for computing the distance between metric representations of graphs in \cite{oles2019efficient}.

The mGH distance retains the expressivity of the standard GH distance --- in particular, its optimal assignments for the example from Figure \ref{fig:uneven_sampling} coincide with those for the GH distance. M\'emoli also proved that the modified and the standard Gromov--Hausdorff distances induce the same topology within $\dGH$-precompact sets of metric spaces. At the same time, Lipschitz equivalence of the Gromov--Hausdorff distances has remained unexplored, although the two are known to be not equal in general.
Interestingly enough, their natural counterparts in the class of ultrametric spaces, first proposed for the GH distance in \cite{zarichnyi2005gromov}, were recently shown to be equal \cite{memoli2019gromov}.

\subsection{Our contribution} The main contribution of this paper is closing the gap about Lipschitz equivalence of the two Gromov--Hausdorff distances. Their equivalence is important to the question of theoretical complexity of computing the mGH distance, due to the fact that the GH distance cannot be tractably approximated to any practical accuracy.

We show that the two distances are equivalent on families of metric spaces whose size is uniformly bounded. Furthermore, we offer a construction proving that the distances are not equivalent in general, even when restricted to ultrametric spaces. Lastly, we study the relation of the Gromov--Hausdorff distances when taken to a regular simplex. We provide conditions under which they are equal, show that they are within a factor of 2 from each other in general, and demonstrate that this bound is in fact tight.

\section{Background}

\subsection{Preliminaries}
The \textit{diameter} of a metric space $X$ is defined as $\displaystyle \diam(X) \defeq \sup_{x,x'\in X} d_X(x, x')$. If $X$ is compact,  $\diam(X) < \infty$.

For $\epsilon \geq 0$, $X_\epsilon \subseteq X$ is an \textit{$\epsilon$-net} of $X$ if $X$ is covered by the union of closed $\epsilon$-balls centered at the points of $X_\epsilon$:$$X \subseteq \bigcup_{x_\epsilon \in X_\epsilon} \{x \in X: d_X(x_\epsilon, x) \leq \epsilon\}.$$ If $X$ is compact, it is guaranteed to have a finite $\epsilon$-net $X_\epsilon$ (that is, $|X_\epsilon| < \infty$) for any $\epsilon>0$.

The \textit{distortion} of a mapping between two metric spaces measures how much it changes the pairwise distances. Let $f : X \to Y$, then $$\dis(f) \defeq \sup_{x, x' \in X} \big|d_X(x, x') - d_Y(f(x), f(x'))\big|.$$

If $f: X \to Y$ and $g: Y \to Z$, then $\dis(g \circ f) \leq \dis(f) + \dis(g)$. This is because, for any $x, x' \in X$, \begin{equation*}
\begin{aligned}
    \big|d_X(x, x') - d_Z((g \circ f)(x), (g \circ f)(x'))\big| &\leq \begin{aligned}[t]
        &\big|d_X(x, x') - d_Y(f(x), f(x'))\big| +\\ &\big|d_Y(f(x), f(x')) - d_Z((g \circ f)(x), (g \circ f)(x'))\big|
    \end{aligned} \\ &\leq \dis(f) + \dis(g).
\end{aligned}
\end{equation*}


For $f: X \to Y$, a mapping $g:f(X)\to X$ is called its \textit{pseudoinverse} if $(f \circ g)(y) = y$ for every $y\in f(X)$, that is, if $g$ maps every point to an element of its pre-image under $f$. It is easy to see that $\dis(g) \leq \dis(f)$ from $$\big|d_Y(y, y') - d_X(g(y), g(y'))\big| = \big|d_X(g(y), g(y')) - d_Y((f \circ g)(y), (f \circ g)(y'))\big| \quad \forall y, y' \in f(X).$$
Trivially, every mapping has a pseudoinverse. If a mapping is injective, its pseudoinverse is unique and coincides with its inverse, which implies that any invertible $f$ has $\dis(f) = \dis(f^{-1})$. Conversely, if a pair of mappings are pseudoinverses of each other, then one is the inverse of another.

The \textit{codistortion} of a pair of mappings $f:X\to Y$ and $g:Y\to X$ is defined as $$\codis(f, g) \defeq \sup_{x\in X, y \in Y} \big|d_X(x, g(y)) - d_Y(f(x), y)\big|,$$ and can be viewed as a measure of how far $f$ and $g$ are from being each other's pseudoinverses, increased by at most the distortion of one of them. Consider the case of $y = f(x)$ to see that $\codis(f, g) \geq \sup_{x\in X} d_X(x, (g \circ f)(x))$, and notice that $\sup_{x\in X} d_X(x, (g \circ f)(x)) = 0$ if and only if $f$ is a pseudoinverse of $g$. At the same time, for any $x \in X$ and $y \in Y$, \begin{align*}
    \big|d_X(x, g(y)) - d_Y(f(x), y)\big| &\leq \begin{aligned}[t]
        &\big|d_X(x, g(y)) - d_Y(f(x), (f \circ g)(y))\big| + \\&\big|d_Y(f(x), (f \circ g)(y)) - d_Y(f(x), y)\big|
    \end{aligned} \\ &\leq \dis(f) + d_Y(y, (f \circ g)(y)),
\end{align*}
so $\codis(f, g) \leq \dis(f) + \sup_{y\in Y} d_Y(y, (f \circ g)(y))$. Applying the symmetric reasoning for the case of $x = g(y)$ and combining the results yields $$\max\big\{\sup_{x\in X} d_X(x, (g \circ f)(x)), \sup_{y\in Y} d_Y(y, (f \circ g)(y))\big\} \leq \codis(f, g)$$ and $$\codis(f, g) \leq \min\big\{\dis(f) + \sup_{y\in Y} d_Y(y, (f \circ g)(y)),\dis(g) + \sup_{x\in X} d_X(x, (g \circ f)(x))\big\}.$$
It follows that $f = g^{-1}$ does not guarantee $\codis(f, g) = 0$, as is demonstrable by any bijection between a pair of non-isometric 2-point spaces. However, $\codis(f, g) = 0$ implies not only $f = g^{-1}$, but also $\dis(f) = \dis(g) = 0$, as for every pair $x, x' \in X$, $$\big|d_X(x, x') - d_Y(f(x), f(x'))\big| = \big|d_X(x, (g \circ f)(x')) - d_Y(f(x), f(x'))\big| \leq \codis(f, g) = 0.$$
\subsection{The Gromov--Hausdorff distance}
Recall that if $X, Y$ are bounded subsets of some metric space $Z$, their proximity can be measured by the Hausdorff distance, defined as $$\dH^Z(X, Y) \defeq \max \big\{\sup_{x \in X} \inf_{y \in Y} d_Z(x, y), \sup_{y \in Y} \inf_{x \in X} d_Z(y, x) \big\} \quad \text{(see Figure \ref{fig:dH})}.$$
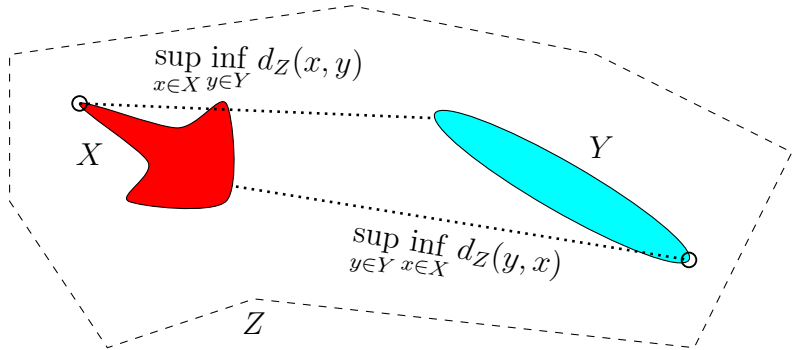
\begin{figure}[H]
    \centering
    \caption{The Hausdorff distance.}
    \label{fig:dH}
    \vspace{.5cm}
    \begin{tikzpicture}[scale=1.3]
\draw[black,dashed] plot coordinates {(-2.5, 1.5)  (-6,1) (-6, -.5) (-5,-2) (-3.5, -1.5) (1,-2) (2, 0) (0, 1) (-2.5, 1.5)} node[black,xshift=-37,yshift=-120] {$Z$};

\begin{scope}[xshift=-65]
    \draw[fill=Xcolor] plot[smooth cycle] coordinates {(-3, .5) (-2, .25) (-1.5, .5) (-1.5,-.5) (-2.5, -.5) (-2.3, -.1)} 
    node[black,xshift=-22,yshift=3] {$X$};
        \node[circle,scale=.5,line width=.25mm] (X_to) at (-1.4, -.35) {};
        \node[circle,draw,scale=.5,line width=.25mm] (X_from) at (-3, .5) {};
\end{scope}
\begin{scope}[xshift=-10,yshift=-10,rotate=-30]
    \draw[fill=Ycolor] (0,0) ellipse (1.5 and .25) node[black,xshift=15,yshift=15] {$Y$};
        \node[circle,draw,scale=.5,line width=.25mm] (Y_from) at (0:1.5 and .25) {};
        \node[circle,scale=.5,line width=.25mm] (Y_to) at (190:1.5 and .25) {};
\end{scope}

\draw[dotted,line width=.35mm] (X_to.center)--(Y_from.center) node[midway,sloped,below]{$\displaystyle \sup_{y \in Y} \inf_{x \in X} d_Z(y, x)$};
\draw[dotted,line width=.35mm] (X_from.center)--(Y_to.center) node[midway,sloped,above]{$\displaystyle \sup_{x \in X} \inf_{y \in Y} d_Z(x, y)$};

\end{tikzpicture}
    \vspace{.1cm}
\end{figure}

The GH distance builds on this notion to measure the difference between compact metric spaces $X$ and $Y$ that do not need to share some ambient metric space. It is defined as $$\dGH(X, Y) \defeq \inf_{\substack{Z\\f_\text{iso}:X\to Z\\g_\text{iso}:Y\to Z}}\dH^Z(f(X), g(Y)),$$ where $Z$ is a metric space and $f_\text{iso}, g_\text{iso}$ are isometries (see Figure \ref{fig:dGH}). Intuitively, $\dGH(X, Y)$ quantifies the difference in shape between $X$ and $Y$ by measuring how close their shape-preserving embeddings can be made to each other. The GH distance is a metric on the space of isometry classes of compact metric spaces, in particular it is 0 if and only if the two spaces are isometric.
\begin{figure}[h]
    \centering
    \caption{The concept behind the Gromov--Hausdorff distance.}
    \label{fig:dGH}
    \vspace{.5cm}
    \begin{tikzpicture}[scale=1.3]
\begin{scope}[xshift=-65]
    \draw[fill=Xcolor] plot[smooth cycle] coordinates {(-3, .5) (-2, .25) (-1.5, .5) (-1.5,-.5) (-2.5, -.5) (-2.3, -.1)} 
    node[black,xshift=-22,yshift=3] {$X$};
\end{scope}
\begin{scope}[xshift=-10,yshift=-10,rotate=-30]
    \draw[fill=Ycolor] (0,0) ellipse (1.5 and .25) node[black,xshift=15,yshift=15] {$Y$};
\end{scope}
\draw[black,dashed] plot coordinates {(-2.5,-.7)  (-6,-1.25) (-4,-4) (-1, -4) (1,-3) (.6, -2.2) (-2.5,-.7)} node[black,above] {$Z$};
\begin{scope}[xshift=-40,yshift=-125, rotate=-50]
    \draw[fill=Xcolor,opacity=.3,text opacity=1] plot[smooth cycle] coordinates {(-3, .5) (-2, .25) (-1.5, .5) (-1.5,-.5) (-2.5, -.5) (-2.3, -.1)} 
    node[black,xshift=-30,yshift=-17,text opacity=1] {$f_\text{iso}(X)$};
\end{scope}
\begin{scope}[xshift=-50,yshift=-80,rotate=20]
    \draw[fill=Ycolor,opacity=.3,text opacity=1] (0,0) ellipse (1.5 and .25) node[black,xshift=23,yshift=-12,text opacity=1] {$g_\text{iso}(Y)$};
\end{scope}
\draw[->] (-4,-.7)--(-3.3,-2.3) node[midway,left]{$f_\text{iso}$};
\draw[->] (-.55,-.7)--(-.9,-2.1) node[midway,right]{$g_\text{iso}$};

\end{tikzpicture}
\end{figure}
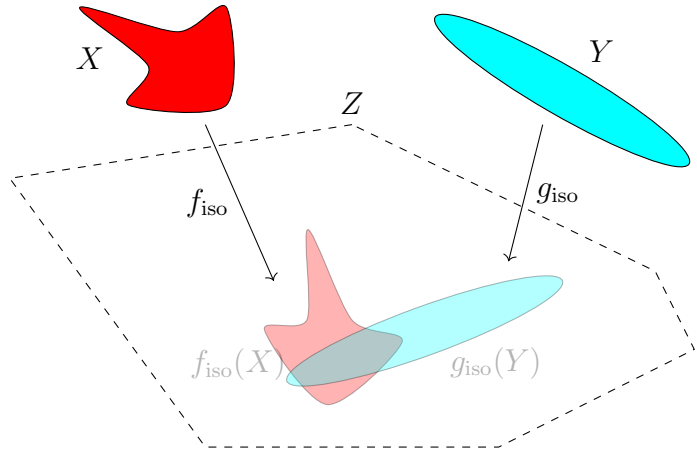

Alternatively, $\dGH(X, Y)$ can be formulated in terms of mappings between $X$ and $Y$ \cite{kalton1999distances}:
\begin{equation*}
    \dGH(X, Y) = \frac{1}{2} \inf_{\substack{f:X\to Y\\g:Y \to X}} \max \{\dis(f), \dis(g), \codis(f, g)\}.
    \label{eqn:alternative dGH}
    \tag{$*$}
\end{equation*}
The alternative formulation \eqref{eqn:alternative dGH} is more computationally friendly, as it allows computing the distance between finite $X$ and $Y$ by searching through finitely many (precisely, $|X|^{|Y|}|Y|^{|X|}$) mapping pairs, as opposed to the infinite number of ambient metric spaces and isometric embeddings into them.


\subsection{The modified Gromov--Hausdorff distance}
The mGH distance is a relaxation of the GH distance obtained by removing the codistortion term $\codis(f, g)$ from \eqref{eqn:alternative dGH},
\begin{equation*}
    \dmGH(X, Y) \defeq \frac{1}{2} \inf_{\substack{f:X\to Y\\g:Y \to X}} \max \{\dis(f), \dis(g)\} = \frac{1}{2}\max\big\{\inf_{f:X\to Y}\dis(f), \inf_{g:Y \to X}\dis(g)\big\},
    \label{eqn:dmGH}
    \tag{$**$}
\end{equation*}
which changes the problem to minimizing the distortion of mappings separately in each direction. Respectively, the size of the search space for $\dmGH(X, Y)$ is $|X|^{|Y|} + |Y|^{|X|}$, which can make a tangible difference from $\dGH(X, Y)$ when the spaces are of small size.

Similarly to the GH distance, the mGH distance is a metric on the space of compact metric spaces (up to isometry). Moreover, the two distances are topologically equivalent within $\dGH$-precompact sets of metric spaces \cite{memoli2012some}. Notice that, by definition, the relaxation is a lower bound of the original GH distance. 

\begin{remark}
Because optimization over singular mappings is studied better than over bi-directional mapping pairs, computing the mGH distance may benefit from the existing techniques. For finite $X$ and $Y$, a mapping $f:X\to Y$ can be represented as a binary row-stochastic matrix $\bm{P} \in \{0, 1\}^{|X| \times |Y|}$, so that the distances in $f(X)$ are contained in $\bm{P}\bm{D}_Y\bm{P}^T$ where $\bm{D}_Y \in \mathbb{R}^{|Y|\times |Y|}$ is the distance matrix of $Y$. This enables matrix representation of $\dis(f)$ as $\|\bm{P}\bm{D}_Y\bm{P}^T\|_\infty$, and therefore casting \eqref{eqn:dmGH} as two minimizations over the space of binary row-stochastic matrices. Subsequently, this search space can be relaxed to its convex hull, the set of all row-stochastic matrices, to enable methods from continuous optimization ---  an approach commonly used in bijection-based shape matching.

\end{remark}

\subsection{Some properties of the Gromov--Hausdorff distances}
The Gromov--Hausdorff distances satisfy
$$\frac{1}{2}|\diam(X) - \diam(Y)| \leq \dmGH(X, Y) \leq \dGH(X, Y) \leq \frac{1}{2}\max\{\diam(X), \diam(Y)\}.$$
For the mGH distance, an additional set of lower bounds can be obtained by comparing the so-called \textit{curvature sets} of $X$ and $Y$ --- the sets of all distance matrices of a given size that are possible in a metric space. This makes the mGH distance more suitable for practical problems and in particular classification tasks, where $\dmGH(X, Y)$ exceeding some threshold is sufficient to conclude that $X$ and $Y$ belong to different classes. We refer the reader to \cite{memoli2012some} and \cite{oles2019efficient} for additional theoretical and computational aspects of these lower bounds.

\begin{claim}
    Let $f:X\to Y$, then $\dGH(X, f(X)) \leq \frac{1}{2}\dis(f)$.
    \label{GH to image}
\end{claim}
\begin{proof}
    Let $g:f(X)\to X$ be a pseudoinverse of $f$, so in particular $\dis(g) \leq \dis(f)$. Moreover, it follows from $y = (f \circ g)(y)$ that $$\codis(f, g) = \sup_{x\in X, y\in f(X)} \big|d_X(x, g(y)) - d_Y(f(x), (f \circ g)(y))\big| \leq \dis(f),$$ and therefore \begin{align*}\dGH(X, f(X)) \leq \frac{1}{2}\max \{\dis(f), \dis(g), \codis(f, g)\} = \frac{1}{2}\dis(f). \tag*{\qedhere} \end{align*}
\end{proof}

\begin{corollary}
Let $f_1, \ldots, f_n$ be a sequence of mappings between compact metric spaces. If $(f_n \circ \ldots \circ f_1)(X)$ is well-defined, then $$\dGH(X, (f_n \circ \ldots \circ f_1)(X)) \leq \frac{1}{2}\sum_{k=1}^n \dis(f_k).$$
\end{corollary}

\begin{claim}
    Let $X_\epsilon$ and $Y_\epsilon$ be $\epsilon$-nets of $X$ and $Y$, respectively, for some $\epsilon \geq 0$. Then $$|\dGH(X, Y) - \dGH(X_\epsilon, Y_\epsilon)| \leq \epsilon$$ and $$|\dmGH(X, Y) - \dmGH(X_\epsilon, Y_\epsilon)| \leq \epsilon.$$
    \label{GH approximation by epsilon-nets}
\end{claim}
\begin{proof}
Consider $f: X \to X_\epsilon$ that maps every point in $X$ to the nearest point in $X_\epsilon$. Trivially, $\dis(f) \leq \epsilon$, and because $X_\epsilon \subseteq X$, $f(X) = X_\epsilon$. From Claim \ref{GH to image}, $\dmGH(X, X_\epsilon) \leq \dGH(X, X_\epsilon) \leq \frac{\epsilon}{2}$. By the analogous reasoning for $Y$, $\dmGH(Y, Y_\epsilon) \leq \dGH(Y, Y_\epsilon) \leq \frac{\epsilon}{2}$. The results follow from the triangle inequalities on $X, X_\epsilon, Y_\epsilon, Y$ for $\dGH$ and $\dmGH$, respectively.
\end{proof}

The above property allows using finite representations of $X$ and $Y$ for either of the Gromov--Hausdorff distances. This is desirable in the computational context, where $\epsilon$-nets $X_\epsilon, Y_\epsilon$ can be used in approximating the distances between potentially infinite $X, Y$ with arbitrary precision (given that the sizes of $X_\epsilon$ and $Y_\epsilon$ for the desired $\epsilon$ are amenable to computation).

\section{Equivalence on spaces of uniformly bounded size}
\label{equivalence}
Let $X_\epsilon$ and $Y_\epsilon$ be finite $\epsilon$-nets of compact metric spaces $X$ and $Y$, respectively, for some $\epsilon \geq 0$. Because the number of mappings between $X_\epsilon$ and $Y_\epsilon$ is finite, there exist $f: X_\epsilon \to Y_\epsilon$ and $g: Y_\epsilon\to X_\epsilon$ s.t. $\frac{1}{2}\max\{\dis(f), \dis(g)\} = \dmGH(X_\epsilon, Y_\epsilon)$. 

Consider a nested sequence of metric spaces $\{X_k\}_{k=0}^\infty$ inductively defined as $X_0 = X_\epsilon$, $X_1 = g(Y_\epsilon)$, and $X_k = (g \circ f)(X_{k-2})$ for $k > 1$. Analogously, define $\{Y_k\}_{k=0}^\infty$ as $Y_0 = Y_\epsilon$, $Y_1 = f(X_\epsilon)$, and $Y_k = (f \circ g)(Y_{k-2})$ for $k > 1$, and notice that $X_{k+1} = g(Y_k)$ for any $k$ (see Figure \ref{fig:nested}).

\begin{figure}[h]
    \centering
    \caption{Nested sequences $\{X_k\}_{k=0}^\infty$ and $\{Y_k\}_{k=0}^\infty$. Cyan and red lines represent acting by $f$ and $g$, respectively.}
    \label{fig:nested}
    \vspace{.5cm}
    \begin{tikzpicture}[node distance=1.75cm,auto,>=stealth']
    \node[] (Xk) {$\{X_k\}_{k=0}^\infty$:};
    \node[below of=Xk, node distance=2cm] (Yk) {$\{Y_k\}_{k=0}^\infty$:};
    \node[right of=Xk, node distance=1.75cm] (X0) {$X_\epsilon$};
    \node[right of=Yk, node distance=1.75cm] (Y0) {$Y_\epsilon$};
    \node[right of=X0] (supsetX01) {$\supseteq$};
    \node[right of=Y0] (supsetY01) {$\supseteq$};
    \node[right of=supsetX01] (X1) {$g(Y_\epsilon)$};
    \node[right of=supsetY01] (Y1) {$f(X_\epsilon)$};
    \node[right of=X1] (supsetX12) {$\supseteq$};
    \node[right of=Y1] (supsetY12) {$\supseteq$};
    \node[right of=supsetX12] (X2) {$(g\circ f)(X_\epsilon)$};
    \node[right of=supsetY12] (Y2) {$(f \circ g)(Y_\epsilon)$};
    \node[right of=X2] (supsetX23) {$\supseteq$};
    \node[right of=Y2] (supsetY23) {$\supseteq$};
    \node[right of=supsetX23] (X3) {$(g \circ f \circ g)(Y_\epsilon)$};
    \node[right of=supsetY23] (Y3) {$(f \circ g \circ f)(X_\epsilon)$};
    \node[right of=X3] (supsetX34) {$\supseteq$};
    \node[right of=Y3] (supsetY34) {$\supseteq$};
    \node[right of=supsetX34, node distance=1.25cm] (Xdots) {$\ldots$};
    \node[right of=supsetY34, node distance=1.25cm] (Ydots) {$\ldots$};

\begin{scope}[transform canvas={xshift=-.1cm}]
    
    \draw[->,color=Xcolor,transform canvas={xshift=-3.5cm}] (Y1) -- (X2) node[above,midway]{};
    \draw[->,color=Ycolor,transform canvas={xshift=-3.5cm}] (X1) -- (Y2) node[above,midway]{};
    \draw[->,color=Xcolor] (Y1) -- (X2) node[above,midway]{};
    \draw[->,color=Ycolor] (X1) -- (Y2) node[above,midway]{};
    \draw[->,color=Xcolor] (Y2) -- (X3) node[above,midway]{};
    \draw[->,color=Ycolor] (X2) -- (Y3) node[above,midway]{};
    \draw[->,color=Xcolor,transform canvas={xshift=7cm}] (Y1) -- (X2) node[above,midway]{};
    \draw[->,color=Ycolor,transform canvas={xshift=7cm}] (X1) -- (Y2) node[above,midway]{};
\end{scope}

\end{tikzpicture}
\end{figure}
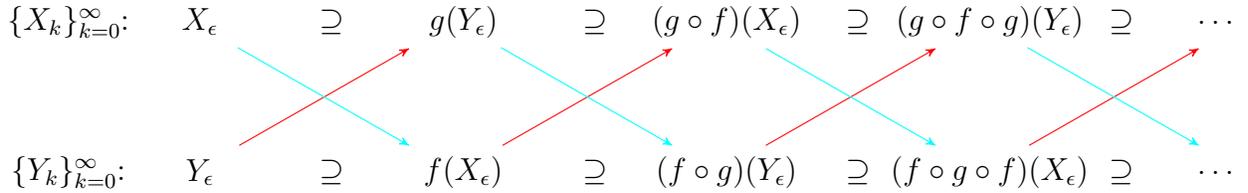

\begin{claim}
Let $\delta \geq 0$. If $\dGH(X_k, X_{k+1}) \leq \delta$ for some $k$, then \begin{align*}\dGH(X, Y) < (2k+1)\dmGH(X, Y) + (2k+2)\epsilon + \delta.\end{align*}
\label{ratio bound}
\end{claim}
\vspace{-1.3cm}
\begin{proof}If $k$ is even then $\dGH(X_\epsilon, X_k), \dGH(Y_\epsilon, Y_k) \leq \frac{k}{2}\max\{\dis(f), \dis(g)\}$ from the corollary to Claim \ref{GH to image}. Analogously, if $k$ is odd, $\dGH(X_\epsilon, Y_k), \dGH(Y_\epsilon, X_k) \leq \frac{k}{2}\max\{\dis(f), \dis(g)\}$. It follows that regardless of the parity of $k$, the triangle inequality on $X_\epsilon, X_k, X_{k+1}, Y_k, Y_\epsilon$ entails
\begin{align*}
    \dGH(X_\epsilon, Y_\epsilon) &\leq \dGH(X_k, X_{k+1}) + \dGH(X_{k+1}, Y_k) + \min\{\begin{aligned}[t]&\dGH(X_\epsilon, X_k) + \dGH(Y_\epsilon, Y_k), \\&\dGH(X_\epsilon, Y_k) + \dGH(Y_\epsilon, X_k)\}\end{aligned} \\
    &\leq \delta + \frac{2k+1}{2}\max\{\dis(f), \dis(g)\} \\
    &= (2k+1)\dmGH(X_\epsilon, Y_\epsilon) + \delta.
\end{align*}
From Claim \ref{GH approximation by epsilon-nets}, $\dGH(X, Y) - \epsilon \leq (2k+1)(\dmGH(X, Y) + \epsilon) + \delta$.
\end{proof}

\begin{corollary}
$\dGH(X, Y) \leq (2n-1)\dmGH(X, Y) + 2n\epsilon$, where $n \defeq \min\{|X_\epsilon|, |Y_\epsilon|\}$.
\end{corollary}
\begin{proof}
Without loss of generality, assume $n = |X_\epsilon|$. Recall that the sequence $\{X_k\}_{k=0}^{\infty}$ is nested and its members are non-empty and finite. Then $\exists k \leq n - 1$ s.t. $X_k = X_{k+1}$ (or otherwise $n = |X_0| > |X_1| > \ldots > |X_n| > 0$ yields a contradiction), and Claim \ref{ratio bound} applies.
\end{proof}

\begin{theorem}
Let $\mathcal{F}$ be a family of metric spaces whose sizes are at most $n < \infty$. Then the Gromov--Hausdorff distances are equivalent on $\mathcal{F}$ with the Lipschitz constant at most $2n-1$.
\end{theorem}
\begin{proof}
For any $X, Y \in \mathcal{F}$, $\dGH(X, Y) \leq (2n-1)\dmGH(X, Y)$ follows from the corollary to Claim \ref{ratio bound} for $\epsilon = 0$.
\end{proof}

To see that the constant $2n-1$ is not tight, it suffices to consider metric spaces of size at most 2.
\begin{claim}
If $|X|, |Y| \leq 2$, then $\dGH(X, Y) = \dmGH(X, Y)$.
\end{claim}
\begin{proof}
Case 1: $\min\{|X|, |Y|\} = 1$. Without loss of generality, let $|Y| = 1$. Then the only possible $f:X \to Y$ has $\dis(f) = \diam(X)$, and therefore $$\frac{1}{2}\diam(X) \leq \dmGH(X, Y) \leq \dGH(X, Y) \leq \frac{1}{2}\max\{\diam(X), \diam(Y)\} = \frac{1}{2}\diam(X).$$

Case 2: $|X| = |Y| = 2$. Then any pair of bijective $f:X \to Y$ and $g = f^{-1}$ satisfies $\dis(f) = \dis(g) = \codis(f, g) = |\diam(X) - \diam(Y)|$, and therefore
\begin{align*}
\dGH(X, Y) \leq \frac{1}{2}|\diam(X) - \diam(Y)| \leq \dmGH(X, Y) \leq \dGH(X, Y). \tag*{\qedhere}
\end{align*}
\end{proof}
\section{Non-equivalence in the general case}
Recall that a metric space $X$ is ultrametric if its triangle inequality is strengthened to $d_X(x, x'') \leq \max\{d_X(x, x'), d_X(x', x'')\} \quad \forall x,x',x'' \in X$, which entails that all triangles in $X$ are isosceles. This section shows that the two Gromov--Hausdorff distances are not equivalent in general, even when restricted to ultrametric spaces. This is demonstrated by constructing pairs of ultrametric spaces that can be mapped into each other with minimal distortion but do not allow for the mappings with small distortion and codistortion at the same time.

Given some $a > 0$, define an operation $\overset{a}{\sqcup}$ between metric spaces so that $X \overset{a}{\sqcup} Y$ produces a pair of set $Z = X \sqcup Y$ (the disjoint union of $X$ and $Y$) and function $d_Z:Z \times Z \to \mathbb{R}$, defined as $d_Z(z, z')=\begin{cases}d_X(z,z') & z,z'\in X\\d_Y(z,z') & z,z'\in Y\\a & \text{otherwise}\end{cases}$.
\begin{claim}
    If $X$ and $Y$ are ultrametric and $a \geq \max\{\diam(X), \diam(Y)\}$, then $X \overset{a}{\sqcup} Y$ is an ultrametric space.
    \label{makes_ultrametric}
\end{claim}
\begin{proof}
Trivially, the strong triangle inequality in $X \overset{a}{\sqcup} Y$ holds for any $x, x', x'' \in X$ or $y, y', y'' \in Y$. Any triangle formed by two points from one metric space and a point from the other one must have two sides of length $a$ and the third side of length at most $a$, thus also satisfying the strong triangle inequality.
\end{proof}
Let $\{U_k\}_{k=0}^\infty$ denote a sequence of metric spaces s.t. $U_0$ is a singleton, $U_1$ is a pair of points at distance 1, and $U_k \defeq U_{k-2} \overset{k}{\sqcup} U_{k-2}$ for any $k > 1$ (see Figure \ref{fig:Uk}). Trivially, $U_0$ and $U_1$ are ultrametric spaces, and by induction using Claim \ref{makes_ultrametric}, so is $U_k$ for every $k > 1$. Notice that $\diam(U_k) = k$ for every $k$ and, moreover, all the points in $U_k$ have identical distance distributions, comprised of positive integers $k, k-2, \ldots$, and 0.

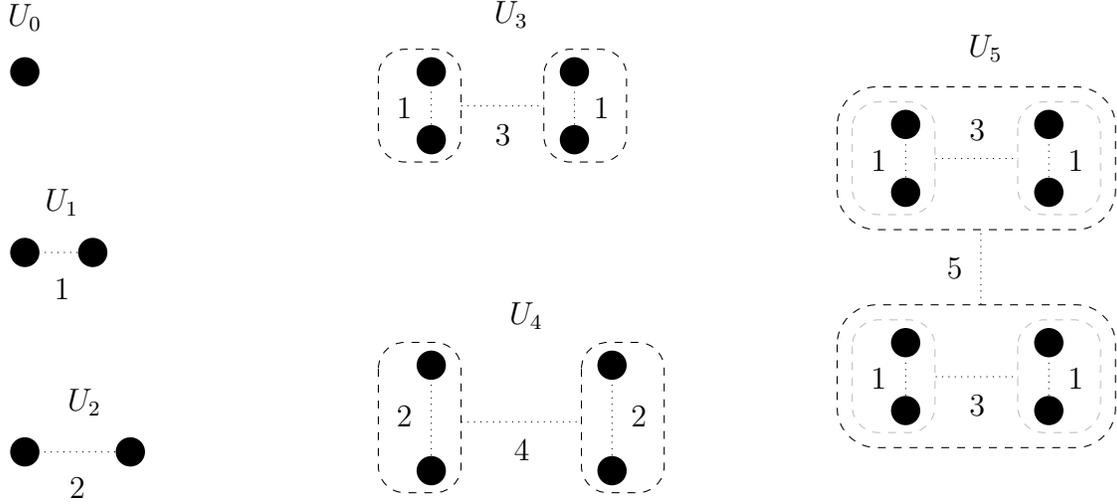
\begin{figure}[h]
    \centering
    \caption{First six members of the ultrametric sequence $\{U_k\}_{k=0}^\infty$. Dashed lines encircle groups of points with identical distances to outside of the group. Numbers by dotted lines indicate distances between points or their groups.}
    \label{fig:Uk}
    \vspace{.5cm}
    \begin{tikzpicture}
\node[circle,fill] (U0){};
\node[above=.2cm of U0]{$U_0$};

\node[circle,fill,below=2cm of U0] (U1a){};
\node[circle,fill,right=.5cm of U1a] (U1b){};
\node[above left=.2cm and -.1cm of U1b]{$U_1$};
\draw[dotted] (U1b) to (U1a) node[below right=.2cm and .25cm]{1};

\node[circle,fill,below=2.25cm of U1a] (U2a){};
\node[circle,fill,right=1cm of U2a] (U2b){};
\node[above left=.2cm and .1cm of U2b]{$U_2$};
\draw[dotted] (U2b) to (U2a) node[below right=.2cm and .45cm]{2};

\node[circle,fill,right=5cm of U0] (U3a){};
\node[circle,fill,below=.5cm of U3a] (U3b){};
\node[circle,fill,right=1.5cm of U3a] (U3c){};
\node[circle,fill,right=1.5cm of U3b] (U3d){};
\node[above right=1.2cm and .55cm of U3b]{$U_3$};
\draw[dotted] (U3b)+(0.4,.45) to ++(1.5,.45) node[below left=.1cm and .3cm]{3};
\draw[dotted] (U3b) to (U3a) node[below left=.2cm and .25cm,xshift=4]{1};
\draw[dashed,rounded corners=10] (4.7,.3) rectangle (5.8,-1.2);
\draw[dotted] (U3d) to (U3c) node[below right=.2cm and .25cm,xshift=-4]{1};
\draw[dashed,rounded corners=10] (6.9,.3) rectangle (8,-1.2);

\node[circle,fill,below=3.5cm of U3a] (U4a){};
\node[circle,fill,below=1cm of U4a] (U4b){};
\node[circle,fill,right=2cm of U4a] (U4c){};
\node[circle,fill,right=2cm of U4b] (U4d){};
\node[above right=1.6cm and .75cm of U4b]{$U_4$};
\draw[dotted] (U4b)+(0.4,.65) -- ++(2,.65) node[below left=.1cm and .55cm]{4};
\draw[dashed,rounded corners=10] (4.7,-3.6) rectangle (5.8,-5.6);
\draw[dotted] (U4b) to (U4a) node[below left=.4cm and .25cm,xshift=4]{2};
\draw[dashed,rounded corners=10] (7.4,-3.6) rectangle (8.5,-5.6);
\draw[dotted] (U4d) to (U4c) node[below right=.4cm and .25cm,xshift=-4]{2};

\node[circle,fill,right=4cm of U3c,yshift=-.7cm] (U5a){};
\node[circle,fill,below=.5cm of U5a] (U5b){};
\node[circle,fill,right=1.5cm of U5a] (U5c){};
\node[circle,fill,right=1.5cm of U5b] (U5d){};
\node[circle,fill,below=2.5cm of U5a] (U5e){};
\node[circle,fill,below=.5cm of U5e] (U5f){};
\node[circle,fill,right=1.5cm of U5e] (U5g){};
\node[circle,fill,right=1.5cm of U5f] (U5h){};
\node[above right=1.45cm and .55cm of U5b]{$U_5$};
\draw[dotted] (U5a)+(1,-1.45) -- ++(1,-2.4) node[below left=-.77cm and .1cm]{5};
\draw[dotted] (U5b)+(0.4,.45) to ++(1.5,.45) node[above left=.1cm and .3cm]{3};
\draw[dotted] (U5b) to (U5a) node[below left=.2cm and .25cm,xshift=4]{1};
\draw[dashed,lightgray,rounded corners=10] (11,-.4) rectangle (12.1,-1.9);
\draw[dotted] (U5d) to (U5c) node[below right=.2cm and .25cm,xshift=-4]{1};
\draw[dashed,lightgray,rounded corners=10] (13.2,-.4) rectangle (14.3,-1.9);
\draw[dashed,rounded corners=15] (10.8,-.2) rectangle (14.5,-2.1);
\draw[dotted] (U5f)+(0.4,.45) to ++(1.5,.45) node[below left=.1cm and .3cm]{3};
\draw[dotted] (U5f) to (U5e) node[below left=.2cm and .25cm,xshift=4]{1};
\draw[dashed,lightgray,rounded corners=10] (11,-3.3) rectangle (12.1,-4.8);
\draw[dotted] (U5h) to (U5g) node[below right=.2cm and .25cm,xshift=-4]{1};
\draw[dashed,lightgray,rounded corners=10] (13.2,-3.3) rectangle (14.3,-4.8);
\draw[dashed,rounded corners=15] (10.8,-3.1) rectangle (14.5,-5);

\end{tikzpicture}
\end{figure}

\begin{claim}
For every $k$, the non-trivial distances from any $x \in U_k$ are all same-parity positive integers up to and including $k$: $\{d_{U_k}(x, x'): x' \neq x\} = \{k - 2j: j = 0, 1, \ldots, \lfloor\frac{k-1}{2}\rfloor\}$.
\label{Uk distances}
\end{claim}
\begin{proof}
The statement is trivial for $U_0$ and $U_1$, and follows for any $U_k$, $k > 1$, from noticing that $U_k$ has distance $k$ in addition to all the distances in $U_{k-2}$ and applying induction.
\end{proof}
An important property of this sequence is that each subsequent member can be mapped to its predecessor without distorting the distances by more than $1$.
\begin{claim}
\label{sequence_map}
For every $k$, there exists $f: U_{k+1} \to U_k$ s.t. $\dis(f) = 1$.
\end{claim}
\begin{proof}
The only possible mapping of $U_1$ into $U_0$ and any bijection between $U_2$ and $U_1$ have the distortion of 1. If for some $k > 1$ there exists $f:U_{k-1}\to U_{k-2}$ s.t. $\dis(f) = 1$, then using $f$ to map both parts of $U_{k+1} = U_{k-1} \overset{k+1}{\sqcup} U_{k-1}$ to the corresponding parts of $U_k = U_{k-2} \overset{k}{\sqcup} U_{k-2}$ also results in the distortion of 1, which concludes the proof by induction.
\end{proof}

Consider two sequences of finite ultrametric spaces $X_k \defeq U_{k+1}$ and $Y_k \defeq U_k \overset{k+2}{\sqcup} U_0$, for $k = 1, 2, \ldots$. The fact that every $Y_k$ is ultrametric follows from Claim \ref{makes_ultrametric}. 
\begin{lemma}
$\dmGH(X_k, Y_k) \leq \frac{1}{2}$ for every $k$.
\label{dmGH sequence}
\end{lemma}
\begin{proof}
From Claim \ref{sequence_map}, there exists $f: X_k \to U_k \subset Y_k$ s.t. $\dis(f) = 1$.

To distinguish the two identical parts of $X_k = U_{k-1} \overset{k+1}{\sqcup} U_{k-1}$, denote them as $U_{k-1}$ and $U'_{k-1}$ respectively. From Claim \ref{sequence_map}, there exists  $g: U_k \to U_{k-1} \subset X_k$ s.t. $\dis(g) = 1$. Let $x' \in U'_{k-1}$, and consider mapping $\widetilde{g}: Y_k \to X_k$, defined as $\widetilde{g}(y)=\begin{cases}g(y) & y \in U_k\\x' & y \in U_0\end{cases}$. Its distortion is \begin{align*}
    \dis(\widetilde{g}) &= \max_{y,y' \in Y_k}\big|d_{Y_k}(y, y') - d_{X_k}(\widetilde{g}(y), \widetilde{g}(y'))\big| \\ &= \max\Big\{\max_{y,y' \in U_k}\big|d_{Y_k}(y, y') - d_{X_k}(g(y), g(y'))\big|, \max_{y \in U_k, y' \in U_0}\big|d_{Y_k}(y, y') - d_{X_k}(g(y), x')\big| \Big\} \\ &= \max \{\dis(g), k+2 - (k+1)\} \\ &= 1,
\end{align*} and therefore $\dmGH(X_k, Y_k) \leq \frac{1}{2}\max\{\dis(f), \dis(\widetilde{g})\} = \frac{1}{2}$.
\end{proof}

\begin{lemma}
$\dGH(X_k, Y_k) \geq \frac{k}{4}$ for every $k$.
\label{dGH sequence}
\end{lemma}
\begin{proof}
Denote the only point in $U_0 \subset Y_k$ as $y_0$, and notice that $d_{Y_k}(y_0, y) = k+2$ for every $y \in Y_k\setminus U_0$. For some $f:X_k \to Y_k$ and $g:Y_k \to X_k$, consider the following possibilities for the pre-image of $y_0$ under $f$.

Case 1: $f^{-1}(y_0) = X_k$. Then $\dis(f) = k + 1 \geq \frac{k}{2}$.

Case 2: $f^{-1}(y_0) = \{x_0\}$. Consider $x \in X_k$ s.t. $d_{X_k}(x, x_0) \in \{1, 2\}$, whose existence is guaranteed by Claim \ref{Uk distances}. From $f(x) \neq y_0$ it follows that $\dis(f) \geq k \geq \frac{k}{2}$.

Case 3: $f^{-1}(y_0) = \emptyset$. Let $x = g(y_0)$, then $\codis(f, g) \geq \big|d_{X_k}(x, g(y_0)) - d_{Y_k}(f(x), y_0)\big| = k+2 \geq \frac{k}{2}$.

Case 4: $1 < |f^{-1}(y_0)| < |X_k|$. Let $x_0 \in X_k$ be the farthest point from $g(y_0)$ s.t. $f(x_0) = y_0$. Claim \ref{Uk distances} entails that $$\exists j \in \{0, 1, \ldots, \lfloor\frac{k}{2}\rfloor\} \quad d_{X_k}(x_0, g(y_0)) = k+1-2j,$$ so in particular $$\codis(f, g) \geq d_{X_k}(x_0, g(y_0)) - d_{Y_k}(f(x_0), y_0) = k-(2j-1).$$
Now, let $x_* \in X_k$ be the nearest point to $g(y_0)$ s.t. $f(x_*) \neq y_0$. Because every $x \in X_k$ s.t. $d_{X_k}(x, g(y_0)) > k+1-2j$ must satisfy $f(x) \neq y_0$, Claim \ref{Uk distances} guarantees that $d_{X_k}(x_*, g(y_0)) \leq k+3-2j$. Then \begin{align*}\codis(f, g) &\geq \big|d_{X_k}(x_*, g(y_0)) - d_{Y_k}(f(x_*), y_0)\big| \\ &= k+2 - d_{X_k}(x_*, g(y_0)) \\ &\geq 2j-1,\end{align*}
and therefore $\codis(f, g) \geq \max\{k-(2j-1), 2j-1\}$. Because $j$ can take a range of values,
\begin{align*}\codis(f, g) &\geq \min_{j = 0, 1, \ldots, \lfloor\frac{k}{2}\rfloor} \max\{k-(2j-1), 2j-1\} \\ &\geq \min_{\lambda \in \mathbb{R}} \max\{k-\lambda, \lambda\} \\ &= \frac{k}{2}.\end{align*}

It follows that $\max\{\dis(f), \codis(f, g)\} \geq \frac{k}{2}$ holds irrespectively of the choice of $f$ and $g$, and therefore \begin{align*}\dGH(X_k, Y_k) \geq \frac{1}{2}\inf_{\substack{f:X_k\to Y_k,\\g:Y_k\to X_k}}\max\{\dis(f), \codis(f,g)\} \geq \frac{k}{4}. \tag*{\qedhere}\end{align*}
\end{proof}

\begin{theorem}
$\dGH$ and $\dmGH$ are not Lipschitz-equivalent even when restricted to ultrametric spaces.
\end{theorem}
\begin{proof}
It follows from Lemmas \ref{dmGH sequence} and \ref{dGH sequence} that $\frac{\dGH(X_k, Y_k)}{\dmGH(X_k, Y_k)} \geq \frac{k}{2}$ for any $k = 1,\ldots,\infty$.
\end{proof}

\section{The Gromov--Hausdorff distances to a regular simplex}
A \textit{regular simplex} is a metric space whose non-trivial distances are all equal. It has been recently shown that certain problems in graph theory and metric geometry can be formulated in terms of the GH distance to a regular simplex. In particular, this notion can be used to calculate the weights of minimum spanning tree edges on a complete graph \cite{ivanov2016geometry}, solve the generalized Borsuk problem, or find clique cover and chromatic numbers of a graph \cite{ivanov2019gromov}.

This section considers the Gromov--Hausdorff distances between a compact metric space $X$ and a regular simplex $\Delta$. It states that $\dmGH(X, \Delta)$ is either equal to $\dGH(X, \Delta)$ or approximates it within the factor of 2, which suggests that the mGH distance can be used for solving or approximating the above problems on par with the GH distance.

\begin{remark}
Notice that the cardinality of $\Delta$ must be finite as long as the Gromov--Hausdorff distances are defined between the compact metric spaces only. However, allowing $\dGH$ and $\dmGH$ to take the value $+\infty$ immediately extends their definitions to also include non-compact spaces, which was the approach taken by Ivanov and Tuzhilin in \cite{ivanov2016geometry} and \cite{ivanov2019gromov}. We note that the results in this section are compatible with the extended definition, and in particular they remain valid if the cardinality of $\Delta$ is allowed to be infinite.
\end{remark}

\begin{theorem}
Let $\Delta$ be a regular simplex, $X$ --- a compact metric space. Then
\begin{enumerate}
    \item $\dGH(X, \Delta) \leq 2\dmGH(X, \Delta)$ is a tight bound;
    \item $\dGH(X, \Delta) = \dmGH(X, \Delta)$ when $|X| < |\Delta|$ or $|X| = |\Delta| < \infty$ or $\diam(\Delta) \leq \frac{1}{2}\diam(X)$.
\end{enumerate}
\end{theorem}
\begin{proof}
See Subsections \ref{simplex of greater cardinality} and \ref{simplex of at most same cardinality}.
\end{proof}

\subsection{Regular simplex of greater cardinality}
\label{simplex of greater cardinality}
Ivanov and Tuzhilin showed in \cite{ivanov2019ultra} that, when $|\Delta| > |X|$, the GH distance between the two spaces is simply $\dGH(X, \Delta) = \frac{1}{2}\max\{\diam(\Delta), \diam(X) - \diam(\Delta)\}.$ Let $f:X\to\Delta$, $g:\Delta\to X$, and consider a sequence $\{(x_k, x'_k)\}_{k=1}^\infty$ in $X \times X$ s.t. $d_X(x_k, x'_k) \to \diam(X)$ as $k \to \infty$.

Case 1: there exists a subsequence $\{(x_{k_j}, x'_{k_j})\}_{j=1}^\infty$ s.t. $f(x_{k_j}) \neq f(x'_{k_j}) \quad \forall j$. Then $\dis(f) \geq |\diam(X) - \diam(\Delta)|$.

Case 2: $f(x_k) = f(x'_k)$ starting from some value of $k$. Then $\dis(f) \geq \diam(X)$.

In any case, $$\dis(f) \geq \min\{|\diam(X) - \diam(\Delta)|, \diam X\} \geq \diam(X) - \diam(\Delta),$$ and, because the choice of $f$ was arbitrary, $\inf_{f:X\to\Delta} \dis(f) \geq \diam(X) - \diam(\Delta).$ At the same time, $|\Delta| > |X|$ means that $g$ cannot be injective and thus $\dis(g) \geq \diam(\Delta)$, which implies $\inf_{g:\Delta\to X} \dis(g) \geq \diam(\Delta).$ It follows that \begin{align*}
    \dmGH(X, \Delta) &= \frac{1}{2}\max\{\inf_{f:X\to\Delta} \dis(f), \inf_{g:\Delta\to X} \dis(g)\}\\ &\geq \frac{1}{2}\max\{\diam(\Delta), \diam(X) - \diam(\Delta)\}\\ &= \dGH(X, \Delta),
\end{align*} and therefore $\dGH(X, \Delta) = \dmGH(X, \Delta)$.

\subsection{Regular simplex of at most the same cardinality}
\label{simplex of at most same cardinality}
Ivanov and Tuzhilin also showed that, when $|\Delta| \leq |X|$, the GH distance is obtained by minimizing the distortion of a surjective $f:X\to \Delta$, $$\dGH(X, \Delta) = \frac{1}{2}\inf_{f(X)=\Delta}\dis(f),$$ and characterized the distortion of mappings from $X$ to $\Delta$ as follows.
\begin{claim}
\label{dis simplex}
For any $f:X\to \Delta$,
\begin{align*}
\dis(f) = \sup\Big(\big\{\diam(f^{-1}(y)): y \in f(X)\big\} \cup \big\{\diam(X) - \diam(\Delta), \diam(\Delta) - \delta(f)\big\}\Big),\end{align*} where $\delta(f) \defeq \inf\{d_X(x, x'): f(x) \neq f(x')\}$.
\end{claim}
\begin{proof}
See Proposition 1.3 in \cite{ivanov2019ultra}.
\end{proof}
Let $\displaystyle \phi \defeq \inf_{f:X\to\Delta}\dis(f)$, $\displaystyle \phi_\Delta \defeq \inf_{f(X)=\Delta}\dis(f)$, and $\displaystyle \gamma \defeq \inf_{g:\Delta\to X}\dis(g)$, so in particular $\dGH(X, \Delta) = \frac{\phi_\Delta}{2}$ and $\dmGH(X, \Delta) = \frac{1}{2}\max\{\phi, \gamma\}$.

Case 1: $\phi_\Delta = \phi$. Then $\dGH(X, \Delta) = \frac{\phi}{2} \leq \dmGH(X, \Delta)$ and so $\dGH(X, \Delta) = \dmGH(X, \Delta)$.

Case 2: $\phi_\Delta > \phi$. For an arbitrary $\epsilon \in (0, \phi_\Delta - \phi)$, consider some $f:X\to \Delta$, $g:\Delta \to X$ s.t. $\dis(f) \leq \phi + \epsilon$ and $\dis(g) \leq \gamma + \epsilon$. Notice that $f$ cannot be surjective, and define $\Delta_0 \defeq \Delta \setminus f(X)$. Construct a surjection $f_\Delta:X\to\Delta$ as follows:
\begin{enumerate}
    \item for every $y \in f(X)$, pick a representative $x_y \in f^{-1}(y)$ (by invoking the axiom of choice, if needed), and denote $X_f \defeq \{x_y: y \in f(X)\}$;
    \item consider some $X_0 \subseteq X \setminus X_f$ of cardinality $|X_0| = |\Delta_0|$, and bijection $f_0:X_0\to\Delta_0$;
    \item set $f_\Delta(x) \defeq \begin{cases}f(x) & x \in X \setminus X_0\\ f_0(x) & x \in X_0
    \end{cases}$ so that $f_\Delta(X\setminus X_0) = f(X)$ and $f_\Delta(X_0) = \Delta_0$.
\end{enumerate}
By construction, $f_\Delta^{-1}(y) \subseteq f^{-1}(y) \quad \forall y \in f(X)$ and $|f_\Delta^{-1}(y)| = 1 \quad \forall y \in \Delta_0$, so $$\sup\big\{\diam(f_\Delta^{-1}(y)): y \in f_\Delta(X)\big\} \leq \sup\big\{\diam(f^{-1}(y)): y \in f(X)\big\}.$$ Because $f_\Delta$ is surjective, $\dis(f_\Delta) \geq \phi_\Delta > \dis(f)$, and from Claim \ref{dis simplex},
\begin{equation*}
    \dis(f_\Delta) = \diam(\Delta) - \delta(f_\Delta) > \diam(X) - \diam(\Delta).
    \label{eqn:diam}
    \tag{$**$$*$}
\end{equation*} By definition, $\delta(f_\Delta) \geq 0$, so the first part of \eqref{eqn:diam} entails \begin{align*}\frac{1}{2}\diam(\Delta) &\geq \frac{1}{2}\dis(f_\Delta)\\ &\geq \dGH(X, f_\Delta(X)) \tag{\text{from Claim \ref{GH to image}}}\\&= \dGH(X, \Delta).\end{align*} Because $f$ is not surjective, $f \circ g:\Delta \to \Delta$ is not surjective either, which means that $(f \circ g)(y) = (f \circ g)(y')$ for some distinct $y, y' \in \Delta$, and so $ \diam(\Delta) \leq \dis(f \circ g)$. Then \begin{align*}\dGH(X, \Delta) &\leq \frac{1}{2}\diam(\Delta)\\ &\leq \frac{1}{2}\dis(f \circ g)\\ &\leq \frac{1}{2}(\dis(f) + \dis(g))\\ &\leq \max\{\phi, \gamma\} + \epsilon \\ &= 2\dmGH(X, \Delta) + \epsilon,
\end{align*} and choosing arbitrarily small $\epsilon$ yields $\dGH(X, \Delta) \leq 2\dmGH(X, \Delta)$.

The second part of \eqref{eqn:diam} implies $\diam(\Delta) > \frac{1}{2}\diam(X)$, so $\diam(\Delta) \leq \frac{1}{2}\diam(X)$ is only possible when $\phi_\Delta = \phi$ and, therefore, $\dGH(X, \Delta) = \dmGH(X, \Delta)$.

\subsubsection{The special case of $|X|=|\Delta|<\infty$ and $\diam(\Delta) > \frac{1}{2}\diam(X)$} Let $\displaystyle \gamma \defeq \inf_{g:\Delta\to X}\dis(g)$ and $\displaystyle \phi_\Delta \defeq \inf_{f(X)=\Delta}\dis(f)$. Trivially, any non-injective $g:\Delta \to X$ has $\dis(g) \geq \diam(\Delta)$. At the same time, $\diam(X) < 2\diam(\Delta)$ entails $$\diam(\Delta) = \max\big\{\diam(\Delta), |\diam(X) - \diam(\Delta)|\big\} \geq \dis(g) \quad \forall g:\Delta \to X,$$ so $\gamma$ must be equal to the infimum of distortions of injective mappings only. Notice that $|X| = |\Delta| < \infty$ means that a mapping in either direction is injective if and only if it is surjective, and recall that any injective $g$ has $\dis(g) = \dis(g^{-1})$. It follows that $\gamma = \phi_\Delta$, which implies $\dmGH(X, \Delta) \geq \frac{\phi_\Delta}{2} = \dGH(X, \Delta)$ and therefore $\dGH(X, \Delta) = \dmGH(X, \Delta)$.

\begin{remark}
The above result, however, does not need to hold when $|X| = |\Delta| \geq |\mathbb{N}|$. Denoting a regular simplex of diameter $\lambda > 0$ as $\Delta_\lambda$, consider $X = \Delta_{\nicefrac{1}{2}} \overset{\nicefrac{1}{4}}{\sqcup} \{x_0\}$ and $\Delta = \Delta_1$ s.t. $|\Delta_{\nicefrac{1}{2}}| = |\Delta_1| \geq |\mathbb{N}|$ (notice that $\diam(\Delta) > \frac{1}{2}\diam(X)$). Furthermore, consider $f:X\to \Delta$ s.t. $f(X) = \{y\}$ for some $y \in \Delta$, and some injective $g:\Delta\to\Delta_{\nicefrac{1}{2}}\subset X$. Trivially, $\dis(f) = \dis(g) = \frac{1}{2}$, and therefore $\dmGH(X, \Delta) \leq \frac{1}{4}$. On the other hand, $\dis(f_\Delta) \geq \frac{3}{4}$ for any surjective $f_\Delta:X\to\Delta$ because there must exist $x \in X$ s.t. $f_\Delta(x) \neq f_\Delta(x_0)$, and by construction $d_X(x, x_0) = \frac{1}{4}$. It follows that $\dGH(X, \Delta) \geq \frac{3}{8} > \frac{1}{4} \geq \dmGH(X, \Delta)$.
\end{remark}

\subsubsection{Tightness of $\dGH(X, \Delta) \leq 2\dmGH(X, \Delta)$}
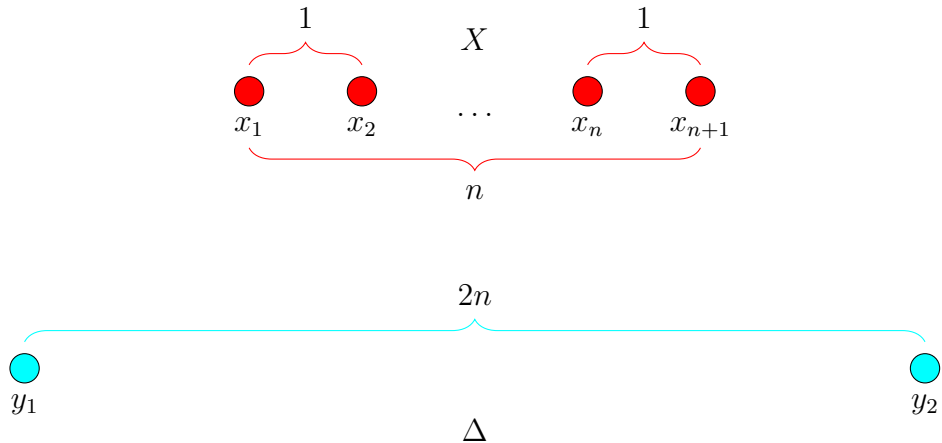
\begin{figure}[!b]
    \centering
    \caption{An example of $X$ and $\Delta$ for which $\frac{\dGH(X, \Delta)}{\dmGH(X, \Delta)}$ can be made arbitrarily close to 2. Numbers by the braces indicate distances between points.}
    \label{fig:tight_bound}
    \vspace{.5cm}
    \begin{tikzpicture}
\node[circle,draw,fill=Xcolor,label=below:{$x_1$},node distance=1.5cm] (X1) at (0,0) {};
\node[circle,draw,fill=Xcolor,right of=X1,label=below:{$x_2$},node distance=1.5cm] (X2) {};
\node[right of=X2,node distance=1.5cm,label=below:{$\ldots$}] (dots) {};
\node[circle,draw,fill=Xcolor,right of=X2,node distance=3cm,label=below:{$x_n$}] (Xn) {};
\node[circle,draw,fill=Xcolor,right of=Xn,label=below:{$x_{n+1}$},node distance=1.5cm] (Xn+1) {};
\node[above of=dots,node distance=.7cm]{$X$};
\draw [Xcolor,decorate,decoration={brace, amplitude=3mm, raise=1.5mm}] (X1.north) -- (X2.north) node[black,pos=0.5, above=5mm]{1};
\draw [Xcolor,decorate,decoration={brace, amplitude=3mm, raise=1.5mm}] (Xn.north) -- (Xn+1.north) node[black,pos=0.5, above=5mm]{1};
\draw [Xcolor,decorate,decoration={brace, amplitude=3mm, mirror, raise=5.5mm}] (X1.south) -- (Xn+1.south)  node[black,pos=0.5, below=9mm]{$n$};

\node[circle,draw,fill=Ycolor,below left=3.4cm and 2.7cm of X1,label=below:{$y_1$}] (Y1) {};
\node[circle,draw,fill=Ycolor,below right=3.4cm and 2.7cm of Xn+1,label=below:{$y_2$}] (Y2) {};
\draw [Ycolor,decorate,decoration={brace, amplitude=3mm, raise=1.5mm}] (Y1.north) -- (Y2.north) node[black,pos=0.5, above=5mm](2n){$2n$};
\node[below of=2n,node distance=1.8cm]{$\Delta$};

\end{tikzpicture}
\end{figure}

To see that the obtained bound is in fact tight, consider the following example, illustrated by Figure \ref{fig:tight_bound}. Let $X = \{x_1, \ldots, x_{n+1}\}$ with $d_X(x_k, x_j) = |k-j|$ and $\Delta = \{y_1, y_2\}$ with $d_\Delta(y_1, y_2) = 2n$. For any surjective $f:X\to\Delta$ there exists some $k = 1, \ldots, n$ s.t. $f(x_k) \neq f(x_{k+1})$, and therefore $\delta(f) = 1$. Using the results of Tuzhilin and Ivanov, \begin{align*}\dGH(X, \Delta) &= \frac{1}{2}\inf_{f(X) = \Delta}\dis(f)\\ &\geq \frac{1}{2}\inf_{f(X)=\Delta} (\diam(\Delta) - \delta(f)) \tag{from Claim \ref{dis simplex}}\\ &= n-\frac{1}{2}.\end{align*} Now, consider $f:X\to\Delta$ s.t. $f(X) = \{y_1\}$ and $g:\Delta\to X$ s.t. $g(y_1) = x_1$, $g(y_2) = x_{n+1}$. Trivially, $\dis(f) = \diam(X) = n$ and $\dis(g) = \diam(\Delta) - \diam(X) = n$, so $\dmGH(X, \Delta) \leq \frac{n}{2}$ and therefore $\frac{\dGH(X, \Delta)}{\dmGH(X, \Delta)} \geq 2 - \frac{1}{n}.$ It follows that $$\lim_{n \to \infty}\frac{\dGH(X, \Delta)}{\dmGH(X, \Delta)} = 2.$$



\bibliographystyle{alpha}
\bibliography{references/references.bib}

\end{document}